\documentclass[12pt]{amsart}

\usepackage{psfrag}
\usepackage{color}

\usepackage{graphicx,graphics}
\usepackage{fullpage,amssymb,amsfonts,amsmath,amstext,amsthm,amscd,enumerate,verbatim,tikz,orcidlink}
\usepackage[T1]{fontenc}
\usetikzlibrary{matrix,arrows}

\usepackage[all]{xy} \usepackage{url}

\begin{document}
\newcommand{\note}[1]{\marginpar{\tiny #1}}
\newtheorem{theorem}{Theorem}[section]
\newtheorem{result}[theorem]{Result}
\newtheorem{fact}[theorem]{Fact}
\newtheorem{conjecture}[theorem]{Conjecture}
\newtheorem{lemma}[theorem]{Lemma}
\newtheorem{proposition}[theorem]{Proposition}
\newtheorem{corollary}[theorem]{Corollary}
\newtheorem{facts}[theorem]{Facts}
\newtheorem{question}[theorem]{Question}
\newtheorem{props}[theorem]{Properties}
\theoremstyle{definition}
\newtheorem{example}[theorem]{Example}
\newtheorem{definition}[theorem]{Definition}
\newtheorem{remark}[theorem]{Remark}

\newcommand{\notes} {\noindent \textbf{Notes.  }}
\renewcommand{\note} {\noindent \textbf{Note.  }}
\newcommand{\defn} {\noindent \textbf{Definition.  }}
\newcommand{\defns} {\noindent \textbf{Definitions.  }}
\newcommand{\x}{{\bf x}}
\newcommand{\z}{{\bf z}}
\newcommand{\B}{{\bf b}}
\newcommand{\V}{{\bf v}}
\newcommand{\T}{\mathcal{T}}
\newcommand{\Z}{\mathbb{Z}}
\newcommand{\Hp}{\mathbb{H}}
\newcommand{\D}{\mathbb{D}}
\newcommand{\R}{\mathbb{R}}
\newcommand{\N}{\mathbb{N}}
\renewcommand{\B}{\mathbb{B}}
\newcommand{\C}{\mathbb{C}}
\newcommand{\dt}{{\mathrm{det }\;}}
 \newcommand{\adj}{{\mathrm{adj}\;}}
 \newcommand{\0}{{\bf O}}
 \newcommand{\w}{\omega}
 \newcommand{\av}{\arrowvert}
 \newcommand{\zbar}{\overline{z}}
 \newcommand{\htt}{\widetilde{h}}
\newcommand{\ty}{\mathcal{T}}
\renewcommand\Re{\operatorname{Re}}
\renewcommand\Im{\operatorname{Im}}
\newcommand{\diam}{\operatorname{diam}}
\newcommand{\dist}{\text{dist}}
\newcommand{\ds}{\displaystyle}
\numberwithin{equation}{section}
\newcommand{\cN}{\mathcal{N}}
\renewcommand{\theenumi}{(\roman{enumi})}
\renewcommand{\labelenumi}{\theenumi}
\newcommand{\inte}{\operatorname{int}}

\newcommand{\dn}[1]{{\scriptsize \color{blue}\textbf{Dan's note:} #1 \color{black}\normalsize}}
\newcommand{\af}[1]{{\scriptsize \color{red}\textbf{Alastair's note:} #1 \color{black}\normalsize}}

\date{\today}

\title{On the ternary Cantor set and Julia sets}
\author{Alastair N. Fletcher\,\orcidlink{0000-0003-1942-6928}}
\email{afletcher@niu.edu}
\address{School of Mathematical and Statistical Sciences, Northern Illinois University, DeKalb, IL 60115-2888, USA}

\begin{abstract}
In this article, we show that the ternary Cantor set is not the Julia set of a rational map. More generally, we show that a Cantor set in the real line cannot be the Julia set of a rational map if the ratios of the lengths of its bounded complementary intervals accumulate only at $0$ and $\infty$. In particular, this applies to every middle-$\lambda$ Cantor set, and to self-similar Cantor sets satisfying the strong separation condition whose contraction ratios are all powers of a common number.
\end{abstract}

\maketitle

\section{Introduction}

In this note, we will assume the reader is familiar with the fundamental ideas from complex dynamics such as Fatou sets, Julia sets and so on, and refer to, for example, the books of Carleson and Gamelin \cite{CG} or Milnor \cite{M}.

Denote by $\mathcal{C}$ the standard ternary Cantor set contained in $[0,1]$. It is well-known that the Julia set of a rational map can be a Cantor set, that is, a set that is homeomorphic to $\mathcal{C}$. For example, if $c$ is a parameter that is not in the Mandelbrot set $\mathcal{M}$, then the Julia set of $z^2+c$ is a Cantor set. More generally, if every critical point of a polynomial is in the escaping set, then its Julia set is a Cantor set \cite[Theorem III.4.2]{CG}.

A more recently initiated theme of research is to study the geometry and topology of Julia sets of uniformly quasiregular mappings in $\R^n$ or $S^n$. The author and Nicks \cite[Theorem 1.1]{FN} showed that every such Julia set is uniformly perfect. Informally, this means that the Julia set isn't separated too much. Moreover, the author and Vellis \cite[Theorem 1.1]{FV} showed that for hyperbolic uniformly quasiregular maps (that is, the Julia set does not meet the post-branch set) for which the Julia set is a Cantor set, the Julia set is uniformly disconnected. Informally, this means that the Julia set does not cluster together too much.

On the other hand, \cite[Theorem 1.2]{FV} shows that every compact, uniformly perfect and uniformly disconnected set in the Riemann sphere $\C_{\infty}$ can be realized as the Julia set of a hyperbolic uniformly quasiregular map $f:\C_{\infty} \to \C_{\infty}$. Of relevance here is the uniformization result of David and Semmes \cite{DS} that states a metric space is quasisymmetrically homeomorphic to $\mathcal{C}$ if and only if it is compact, doubling, uniformly disconnected and uniformly perfect. Thus $\mathcal{C}$ {\bf is} the Julia set of a uniformly quasiregular map.

It is natural to then ask if we can take this uniformly quasiregular mapping to be a map with no distortion. In other words, is $\mathcal{C}$ the Julia set of a rational map? The point of this note is to answer this question in the negative, which in turn gives support for using the dynamics of uniformly quasiregular mappings as the natural setting for asking classification questions on the geometry of Cantor sets which are also Julia sets. 

\begin{theorem}
\label{thm:1}
There is no rational map $f$ for which $J(f)$ is $\mathcal{C}$.
\end{theorem}

As far as the author is aware, Theorem \ref{thm:1} has not been addressed in the literature before. In fact, Theorem \ref{thm:1} is a special case of a more general statement, which uses only the lengths of the complementary intervals of the Cantor set. Let $X\subset \R$ be a Cantor set, so that $X$ is non-empty, compact, perfect and totally disconnected. By a {\it gap} of $X$ we mean a bounded component of $\R \setminus X$, and we write $\ell(G)$ for the length of a gap $G$. The {\it ratio set} of $X$ is
\[ \mathcal{R}(X) = \left \{ \frac{\ell(G)}{\ell(G')} : G, G' \text{ gaps of } X \right \} \subset (0,\infty). \]
We say that $X$ has {\it discrete gap ratios} if
\begin{equation}
\label{eq:star}
\mathcal{R}(X)\cap [a,b] \text{ is finite whenever } 0<a<b<\infty .
\end{equation}

\begin{theorem}
\label{thm:2}
Let $X\subset \R$ be a Cantor set with discrete gap ratios. Then there is no rational map $f$ for which $J(f)$ is $X$.
\end{theorem}

The gaps of $\mathcal{C}$ have lengths $3^{-k}$ for $k\in\N$, so $\mathcal{R}(\mathcal{C}) = \{3^n : n\in\Z\}$ and Theorem \ref{thm:1} follows from Theorem \ref{thm:2}. The same applies to the other middle-$\lambda$ Cantor sets. For $\lambda\in(0,1)$, let $\mathcal{C}_\lambda$ be the middle-$\lambda$ Cantor set, that is, the attractor of the pair of similarities $x\mapsto rx$ and $x\mapsto rx + 1-r$, where $r = (1-\lambda)/2$. Thus $\mathcal{C} = \mathcal{C}_{1/3}$.

\begin{corollary}
\label{cor:1}
For every $\lambda\in(0,1)$, there is no rational map $f$ for which $J(f)$ is $\mathcal{C}_\lambda$.
\end{corollary}

Corollary \ref{cor:1} is a special case of Proposition \ref{prop:1} below, which shows that \eqref{eq:star} holds for self-similar Cantor sets in $\R$ whose contraction ratios are all powers of a common number.

Another natural question concerns the maximal dilatation of the uniformly quasiregular maps which have $\mathcal{C}$ as their Julia set. For a uniformly quasiregular map $f$, denote by $K(f)$ the smallest $K\geq 1$ for which every iterate of $f$ is $K$-quasiregular, and set
\[ K(\mathcal{C}) = \inf \{ K(f) : f:\C_{\infty} \to \C_{\infty} \text{ is uniformly quasiregular and } J(f) = \mathcal{C} \}. \]
As observed above, this infimum is taken over a non-empty set. Since a $1$-quasiregular map of $\C_{\infty}$ is a rational map, Theorem \ref{thm:1} shows that there is no $f$ in this family with $K(f)=1$. It does not by itself show that $K(\mathcal{C})>1$, since we do not know whether the infimum is attained.

\begin{question}
\label{q:1}
Is $K(\mathcal{C})>1$? Is the infimum in the definition of $K(\mathcal{C})$ attained?
\end{question}

It is well-known that every uniformly quasiregular map of $\C_{\infty}$ is quasiconformally conjugate to a rational map, see for example \cite{Hi}. Consequently, $K(\mathcal{C}) = 1$ would mean, roughly speaking, that $\mathcal{C}$ is the image of Julia sets of rational maps under quasiconformal maps of $\C_{\infty}$ with dilatation arbitrarily close to $1$.

One could ask if there is a geometric property of the Cantor set which would yield the analogous quantity $K(X)$ for other uniformly perfect and uniformly disconnected Cantor sets $X$? A related question is addressed in the work of Shiga \cite{S} on moduli spaces of Cantor sets, that is, when two Cantor sets are ambiently quasiconformally equivalent.

It is also natural to ask if Theorem \ref{thm:1} holds for all Cantor sets which arise as the attractor set of an iterated function system generated by similarities. Theorem \ref{thm:2} gives a partial answer to this question, and the following cases remain open.

\begin{question}
\label{q:2}
Let $X\subset\R$ be a self-similar Cantor set as in Proposition \ref{prop:1}, but whose contraction ratios are not all powers of a common number. Can $X$ be the Julia set of a rational map? What if $X\subset\C$ is a self-similar Cantor set that is not contained in a line?
\end{question}

Finally here, we point out that there is no transcendental entire function $f$ with $J(f)$ totally disconnected, as Baker \cite{B} showed that $J(f)$ must contain a continuum in this setting.

The author wishes to thank Vyron Vellis for many conversations on topics relating to those covered in this note. The statement of Theorem \ref{thm:1} was first given in an online talk of the author in the {\it Mostly Teichm\"uller Spaces} seminar in 2021 - the author thanks the organiser, Dragomir Saric, for the invitation to speak at this seminar. The proofs in Section \ref{sec:2} follow an argument suggested by the anonymous referee for Theorem \ref{thm:1}, which is considerably simpler than the author's original proof. The author thanks the referee for this and for their other comments.

\section{Proof of Theorem \ref{thm:2}}
\label{sec:2}

Throughout this section, $X\subset\R$ is a Cantor set and $f$ is a rational map of degree $d\geq 2$ with $J(f) = X$. Recall that the Julia set and its complement, the Fatou set, are completely invariant under $f$, see for example \cite{M}. In particular $f(X) = X$.

\begin{lemma}
\label{lem:1}
We have $f(\R)\subset \R\cup\{\infty\}$, and $f$ has no poles in $X$. Consequently $f'$ is real-valued on $\R$ away from the poles of $f$.
\end{lemma}

\begin{proof}
The map $f^*(z) = \overline{f(\zbar)}$ is rational. If $x\in X$, then $f(x) \in X \subset \R$ and so $f^*(x) = f(x)$. Since $X$ is an infinite compact subset of $\R$, the Identity Theorem gives $f^* = f$ on $\C_{\infty}$, and hence $f(\R)\subset\R\cup\{\infty\}$. Finally, $\infty\notin X = f(X)$, so $f$ has no poles in $X$.
\end{proof}

Each gap of $X$ is an open interval $(s,t)$ with $s,t\in X$, and distinct gaps are disjoint. We call a gap $G$ {\it good} if $G$ contains neither a pole nor a critical point of $f$. Since $f$ has finitely many poles and critical points, all but finitely many gaps of $X$ are good.

\begin{lemma}
\label{lem:2}
If $G$ is a good gap of $X$, then $f(G)$ is a gap of $X$.
\end{lemma}

\begin{proof}
Write $G = (s,t)$. By Lemma \ref{lem:1}, $f$ is real-valued and continuous on $[s,t]$. As $f'$ does not vanish on $(s,t)$, $f$ is strictly monotonic on $[s,t]$, and so $f(G)$ is the bounded open interval with endpoints $f(s)\neq f(t)$. These endpoints lie in $X$ because $s,t\in X$. On the other hand, $G$ is contained in the Fatou set of $f$, and hence so is $f(G)$, that is, $f(G)\cap X = \emptyset$. A bounded open interval which is disjoint from $X$ and has both endpoints in $X$ is a component of $\R\setminus X$.
\end{proof}

\begin{lemma}
\label{lem:3}
Suppose that $X$ has discrete gap ratios. Then for every $x\in X$,
\[ f'(x) \in T := \mathcal{R}(X) \cup \bigl(-\mathcal{R}(X)\bigr) \cup \{0\}. \]
\end{lemma}

\begin{proof}
Fix $x\in X$. For each $n\in\N$, since $X$ is perfect there exists $y_n\in X$ with $0<|x-y_n|<1/n$, and since $X$ is totally disconnected the open interval with endpoints $x$ and $y_n$ contains a point of $\R\setminus X$, and hence contains a gap of $X$. Any given gap arises in this way for only finitely many $n$. We therefore obtain a sequence of distinct gaps $G_n = (s_n,t_n)$ with $s_n\to x$ and $t_n\to x$, and by discarding finitely many of them we may assume that every $G_n$ is good.

By Lemma \ref{lem:1}, $f$ is holomorphic in a neighbourhood of $x$, and so
\[ q_n := \frac{f(t_n)-f(s_n)}{t_n-s_n} = \frac{1}{t_n-s_n}\int_{s_n}^{t_n} f'(u)\,du \longrightarrow f'(x) \]
as $n\to\infty$. By Lemma \ref{lem:2}, $|f(t_n)-f(s_n)|$ is the length of the gap $f(G_n)$, and so $q_n \in \mathcal{R}(X)\cup(-\mathcal{R}(X))$ for every $n$.

If $f'(x) = 0$ there is nothing to prove. Otherwise, put $a = |f'(x)|/2$ and $b = 2|f'(x)|$. For all large $n$ we have $a\leq |q_n|\leq b$, and by \eqref{eq:star} there are only finitely many elements of $\mathcal{R}(X)\cup(-\mathcal{R}(X))$ with modulus in $[a,b]$. A convergent sequence taking values in a finite set is eventually constant, and therefore $f'(x) = q_n$ for all large $n$. In particular $f'(x)\in T$.
\end{proof}

\begin{proof}[Proof of Theorem \ref{thm:2}]
Suppose, for a contradiction, that $X$ has discrete gap ratios and that $f$ is a rational map of degree $d\geq 2$ with $J(f)=X$. Since $f$ has at most $2d-2$ critical points and $X$ is infinite, we may choose $x_0\in X$ with $f'(x_0)\neq 0$. Since $X$ is perfect, there is a sequence $(x_n)_{n=1}^{\infty}$ in $X\setminus\{x_0\}$ with $x_n\to x_0$.

By Lemma \ref{lem:3}, $f'(x_n)\in T$ for every $n$, and $f'(x_n)\to f'(x_0)\neq 0$ by continuity. Exactly as in the proof of Lemma \ref{lem:3}, \eqref{eq:star} implies that for large $n$ the values $f'(x_n)$ lie in a finite set, and hence $f'(x_n) = f'(x_0)$ for all large $n$. Thus the rational function $f' - f'(x_0)$ vanishes on a set with an accumulation point at $x_0$, which is not a pole of $f$. By the Identity Theorem, $f'$ is constant. Hence $f$ is linear, but the Julia set of a linear map cannot be $X$ and this contradiction completes the proof.
\end{proof}

\begin{remark}
\label{rem:2}
The proof shows a little more. Without assuming \eqref{eq:star}, the argument of Lemma \ref{lem:3} gives $f'(X)\subset \overline{T}$, where the closure is taken in $\R$. As $f$ has degree at least two, $f'$ is non-constant, and so we may choose $x_0\in X$ with $f''(x_0)\neq 0$. Then $f'$ is injective on a real neighbourhood $U$ of $x_0$, and $X\cap U$ is uncountable. Hence, if a Cantor set $X\subset\R$ is the Julia set of a rational map, then $\overline{\mathcal{R}(X)}$ is uncountable. In fact $\overline{T}$ contains the image of a relatively open subset of $X$ under a real-analytic diffeomorphism. In particular, the conclusion of Theorem \ref{thm:2} holds whenever $\overline{\mathcal{R}(X)}$ is countable.
\end{remark}

\section{Self-similar Cantor sets}

Let $m\geq 2$ and, for $i=1,\ldots,m$, let $\varphi_i(x) = \varepsilon_i r_i x + b_i$ be a similarity of $\R$, where $\varepsilon_i\in\{-1,1\}$, $0<r_i<1$ and $b_i\in\R$. Let $X$ be the attractor of $\{\varphi_1,\ldots,\varphi_m\}$, that is, the unique non-empty compact set with $X = \bigcup_{i=1}^m\varphi_i(X)$, and let $I$ be the convex hull of $X$. We say that the {\it strong separation condition} holds if the intervals $\varphi_1(I),\ldots,\varphi_m(I)$ are pairwise disjoint. In this case it is standard that $X$ is a Cantor set.

\begin{proposition}
\label{prop:1}
Let $X$ be the attractor of similarities $\varphi_1,\ldots,\varphi_m$ of $\R$ satisfying the strong separation condition. Suppose that there exist $\mu\in(0,1)$ and $k_1,\ldots,k_m\in\N$ such that $r_i = \mu^{k_i}$ for $i=1,\ldots,m$. Then $X$ has discrete gap ratios, and hence $X$ is not the Julia set of a rational map.
\end{proposition}

\begin{proof}
The complement of $\bigcup_{i=1}^m\varphi_i(I)$ in $I$ consists of $m-1$ open intervals $H_1,\ldots,H_{m-1}$, with lengths $h_1,\ldots,h_{m-1}$. Since the endpoints of $I$ lie in $X$, the endpoints of each $\varphi_i(I)$ lie in $X$. As $X\subset\bigcup_{i=1}^m\varphi_i(I)$, it follows that each $H_j$ is a gap of $X$.

We claim that every gap $G$ of $X$ is of the form $\varphi_{i_1}\circ\cdots\circ\varphi_{i_p}(H_j)$ for some $p\geq 0$, some $i_1,\ldots,i_p\in\{1,\ldots,m\}$ and some $j$. Write $G=(s,t)$, so that $s,t\in X$ and $G\subset I$. Suppose first that $s$ and $t$ lie in different intervals $\varphi_i(I)$. Since the endpoints of these intervals lie in $X$, none of them lies in $G$. Hence $s$ is the right endpoint of one such interval, $t$ is the left endpoint of the next one to the right, and $G = H_j$ for some $j$. Otherwise, $s,t\in\varphi_i(I)$ for some $i$. By the strong separation condition, $X\cap\varphi_i(I) = \varphi_i(X)$, and so $\varphi_i^{-1}(G)$ is an open interval which is disjoint from $X$ and has endpoints in $X$, that is, a gap of $X$, of length $\ell(G)/r_i$. We now repeat the argument with $\varphi_i^{-1}(G)$ in place of $G$. Each step increases the length by a factor of at least $1/\max_i r_i > 1$, while every gap has length at most that of $I$, and so the process terminates. This proves the claim.

It follows that every gap of $X$ has length $h_j\mu^k$ for some $j\in\{1,\ldots,m-1\}$ and some integer $k\geq 0$. Hence
\[ \mathcal{R}(X) \subset \left \{ \frac{h_j}{h_{j'}}\,\mu^n : 1\leq j,j'\leq m-1, \ n\in\Z \right \}, \]
which is a finite union of geometric progressions, each of which meets any interval $[a,b]$ with $0<a<b<\infty$ in a finite set. Thus \eqref{eq:star} holds, and the final assertion follows from Theorem \ref{thm:2}.
\end{proof}

Corollary \ref{cor:1} is the case $m=2$, $r_1=r_2 = (1-\lambda)/2$ of Proposition \ref{prop:1}.

\begin{remark}
\label{rem:3}
The hypothesis on the contraction ratios in Proposition \ref{prop:1} cannot be dropped from the argument. Suppose that the strong separation condition holds and that there exist $i\neq j$ for which $\log r_i/\log r_j$ is irrational. For all integers $a,b\geq 0$, the images of $H_1$ under the iterates $\varphi_i^{a}$ and $\varphi_j^{b}$ are gaps of $X$, and so $r_i^{a}r_j^{-b}\in\mathcal{R}(X)$. These numbers are dense in $(0,\infty)$. Hence $\overline{\mathcal{R}(X)} = [0,\infty)$, and neither Theorem \ref{thm:2} nor Remark \ref{rem:2} gives any information. This is the first case of Question \ref{q:2}.
\end{remark}

\end{document}